\documentclass[reqno]{amsart}
\usepackage[top=1.3in, bottom=1.1in, left=1.35in, right=1.35in, marginpar=1in]{geometry}
\usepackage{amsmath,amsthm,amssymb,amsfonts,extarrows,bm,mathrsfs,tcolorbox,bbm,color,soul,xcolor,enumitem,framed,pdfpages}
\usepackage[abs]{overpic}
\usepackage[titletoc]{appendix}
\usepackage{mathtools}
\usepackage[utf8]{inputenc}
\usepackage[english]{babel}
\usepackage{float}
\usepackage{soul}
\usepackage{tikz-cd}
\usepackage{wasysym}
\usepackage{caption}
\usepackage{graphicx} 
\usepackage{subcaption}

\usepackage[style=alphabetic]{biblatex} 
\usepackage{csquotes}

\usepackage{hyperref}
\hypersetup{
    colorlinks=true,
    linkcolor=magenta,
    filecolor=cyan,      
    urlcolor=blue,
}
 
\usepackage{hyperref}

\newtheorem*{definition*}{Definition}
\newtheorem{definition}{Definition}[section]
\newtheorem*{theorem*}{Theorem}
\newtheorem{theorem}[definition]{Theorem}
\newtheorem*{lemma*}{Lemma}
\newtheorem{lemma}[definition]{Lemma}
\newtheorem*{proposition*}{Proposition}

\newtheorem*{example*}{Example}

\newtheorem*{corollary*}{Corollary}

\newtheorem*{conjecture*}{Conjecture}
\newtheorem{conjecture}[definition]{Conjecture}
\newtheorem*{question*}{Question}
\newtheorem{question}[definition]{Question}
\theoremstyle{remark}
\newtheorem*{remark*}{Remark}
\newtheorem{remark}[definition]{Remark}

\title{A Note on Knot Floer Homology of Satellite Knots with (1, 1)-Patterns}
\author[W.\ Shen]{Weizhe Shen}
\address {School of Mathematics, Georgia Institute of Technology, Atlanta, GA 30332}
\email{wshen41@gatech.edu}

\begin{document}
\maketitle

\begin{abstract}
	We prove that if $P$ is a $(1,1)$-pattern knot, the two inequalities $\dim \widehat{HFK} (P(K)) \geqslant \dim \widehat{HFK} (P(U))$ and $\dim \widehat{HFK} (P(K)) \geqslant \dim \widehat{HFK} (K)$ hold for the unknot $U\subset S^3$ and any companion knot $K\subset S^3$.
\end{abstract}

\section{Introduction}
Knot Floer homology, introduced
independently by Ozsv{\'a}th-Szab{\'o} \cite{ozsvath2004holomorphic} and J. Rasmussen \cite{rasmussen2003floer}, is a powerful invariant of knots in the three-sphere. For example, it captures several geometric properties of knots such as genus \cite{ozsvath2004holomorphicgenus} and fiberedness \cite{ghiggini2008knot,ni2007knot}. The theory has several different variants; in this note, we assume the reader is familiar with the hat version, which takes the form of a bi-graded finitely generated vector space over the field $\mathbb{F}:=\mathbb{Z}/2\mathbb{Z}$:
\[
    \widehat{HFK}(K) = \bigoplus_{m,a \in \mathbb{Z}} \widehat{HFK}_m (K,a),
\]
where $K$ is a knot in $S^3$, $m$ is the \textit{Maslov} (or \textit{homological}) grading, and $a$ is the \textit{Alexander} grading.

To three-manifolds with parameterized boundary, Lipshitz-Ozsv{\'a}th-Thurston \cite{lipshitz2018bordered} associated bordered Heegaard Floer invariants. Moreover, their pairing theorems are well-adapted to the study of the result of gluing two manifolds with torus boundary. Recall that given a (pattern) knot $P$ embedded in a standard solid torus $S^1 \times D^2=:V$ and a (companion) knot $K$ in $S^3$, the satellite knot $P(K)$ is obtained from $P$ by gluing $V$ to the complement $X_K:=\overline{S^3 - \nu(K)}$ (where $\nu(K)$ is a tubular neighborhood of $K$) in such a way that the meridian of $V$ is identified with the meridian of $K$, and the longitude of $V$ is identified with the Seifert longitude of $K$. Therefore, satellite knots can be studied using bordered Heegaard Floer homology. Some early work in this approach includes \cite{levine2012knot, petkova2013cables, hom2014bordered}. This note concerns satellite knot with $(1,1)$-patterns; a knot $P\subset S^1\times D^2$ is called a \textit{(1,1)-pattern} if it admits a genus-one doubly-pointed bordered Heegaard diagram.

For three-manifolds with a single toroidal boundary component, Hanselman-Rasmussen-Watson \cite{hanselman2017bordered,hanselman2018heegaard} interpreted the relevant bordered Heegaard Floer invariants geometrically as decorated immersed curves in the once-punctured
torus. Later, a formula for the behavior of these immersed curves under cabling was given in Hanselman-Watson \cite{hanselman2019cabling}. More recently, Chen \cite{chen2019knot} studied the computation of knot Floer chain complexes of satellite knots with $(1,1)$-patterns by using immersed curves. 

Does a non-zero degree map give a rank inequality on Heegaard Floer homology? 
More specifically, Hanselman-Rasmussen-Watson \cite[Question 12]{hanselman2017bordered} asked, if there is a degree-one map $Y_1 \to Y_2$ between closed, connected, orientable three-manifolds, is it the case that $\dim \widehat{HF}(Y_1) \geqslant \dim \widehat{HF}(Y_2)$? For integer homology spheres, Karakurt-Lidman \cite[Conjecture 9.4]{karakurt2015rank} proposed that if there is a non-zero degree map $Y_1 \to Y_2$ between them, then $\text{rank}\,HF_{red}(Y_1) \geqslant \text{rank}\,HF_{red}(Y_2)$ and $\text{rank}\,\widehat{HF}(Y_1) \geqslant \text{rank}\,\widehat{HF}(Y_2)$. Karakurt-Lidman \cite[Theorem 1.9]{karakurt2015rank} also studied maps between Seifert homology spheres. 
It is natural to ask similar questions about knot complements in $S^3$ and rank inequalities on knot Floer homology. 
Given a degree-one map $\varphi: X_K \to X_U$ that preserves peripheral structure\footnote{We refer the reader to \cite[Proposition 1]{boileau2016one} for a proof of the existence of such a map.}, the induced map $\tilde{\varphi}: X_{P(K)} \to X_{P(U)}$ is well-defined and further induces an epimorphism $\tilde{\varphi}_*: \pi_1(X_{P(K)}) \to \pi_1(X_{P(U)})$ that also preserves peripheral structure. This is a special case of \cite[Question 1.9]{juhasz2016concordance}, in which Juh\'{a}sz-Marengon asked, 
for knots $K_1$ and $K_2$ in $S^3$ such that there is an epimorphism $\pi_1(X_{K_1}) \to \pi_1(X_{K_2})$ preserving peripheral structure,
is it true that $\dim \widehat{HFK}(K_1) \geqslant \dim \widehat{HFK} (K_2)$? 
We state the version corresponding to the special case $\pi_1(X_{P(K)}) \to \pi_1(X_{P(U)})$ in the following conjecture.

\begin{conjecture}\label{conj:special}
    Given any pattern knot $P$ in $S^1\times D^2$ and any companion knot $K$ in $S^3$, there is an inequality $\dim \widehat{HFK} (P(K)) \geqslant \dim \widehat{HFK} (P(U))$, where $U$ denotes the unknot in $S^3$.
\end{conjecture}

Another closely related conjecture is the following:

\begin{conjecture}\label{conj:dim}
    Given any pattern knot $P$ in $S^1\times D^2$ and any companion knot $K$ in $S^3$, there is an inequality $\dim \widehat{HFK} (P(K)) \geqslant \dim \widehat{HFK} (K)$.
\end{conjecture}

The purpose of this note is to prove Conjectures \ref{conj:special} and \ref{conj:dim} when $P$ is a $(1,1)$-pattern, by using (geometrically interpreted) bordered Heegaard Floer invariants. 

\begin{theorem}\label{thm:second}
    For any $(1,1)$-pattern knot $P \subset S^1\times D^2$, the inequality
    \begin{equation}\label{eqn: pattern}
        \dim \widehat{HFK} (P(K)) \geqslant \dim \widehat{HFK} (P(U))
    \end{equation}
    holds for the unknot $U\subset S^3$ and any companion knot $K\subset S^3$.
\end{theorem}

\begin{theorem}\label{thm:main}
    Given a $(1,1)$-pattern knot $P\subset S^1\times D^2$, the inequality
    \begin{equation}\label{eqn: companion}
        \dim \widehat{HFK} (P(K)) \geqslant \dim \widehat{HFK} (K)
    \end{equation}
    holds for any companion knot $K\subset S^3$. 
\end{theorem}

Two natural questions (which were originally pointed out by Tye Lidman) to ask are the following:

\begin{question}
Is it possible to characterize the conditions of equality and strict-inequality for the two inequalities (\ref{eqn: pattern}) and (\ref{eqn: companion}), respectively?
\end{question}

\begin{question}
Does either of the above theorems have a refinement for Maslov gradings?
\end{question}

Although these two questions are not fully resolved here, we discuss them in detail in Section \ref{sec: FurtherRemarks}; in particular, we will also see that there is no refinement for Alexander gradings.

We conclude this section by briefly explaining two major parts in proving the above theorems. To begin with, the main result of \cite{chen2019knot} allows us to obtain the dimension of $\widehat{HFK} (P(K))$ by counting the minimum intersections of the curves $\beta(P)$ and $\alpha(K)$ in its corresponding pairing diagram (which is defined in Section \ref{sec:background}):

\begin{theorem}[\cite{chen2019knot}, Theorem 1.2]\label{thm:wchen}
    Given a (1,1)-pattern knot $P\subset S^1\times D^2$ and a companion knot $K\subset S^3$, let $\widehat{HF}(X_K) \subset \partial X_K \backslash \{w'\}$ be the immersed curves of the knot complement $X_K$, and let $(\beta, \mu, \lambda, w, z) \subset \partial (S^1\times D^2)$ be a 5-tuple corresponding to a genus-one doubly-pointed bordered Heegaard diagram for $P$. Let $h:\partial X_K \to \partial (S^1\times D^2)$ be an orientation preserving homeomorphism such that
    \begin{enumerate}
        \item $h$ identifies the meridian and Seifert longitude of $K$ with $\mu$ and $\lambda$, respectively;
        \item $h(w') = w$;
        \item there is a regular neighborhood $U\subset \partial (S^1\times D^2)$ of $w$ such that $z\in U$, $U\cap (\lambda\cup \mu) = \emptyset$, and $U\cap h(\widehat{HF}(X_K) ) = \emptyset$.
    \end{enumerate}
    Let $\alpha = h(\widehat{HF}(X_K))$. Then there is a chain homotopy equivalence
    \[
        \widehat{CFK}(\alpha, \beta, w,z) \cong \widehat{CFK}(S^3, P(K)).
    \]
    Moreover, if $\alpha$ is connected, this chain homotopy equivalence preserves the Maslov grading and Alexander filtration.
\end{theorem}

Another essential part of the proof—applying a sequence of moves to curves without increasing intersection number—is inspired by the proof of \cite[Theorem 52]{hanselman2017bordered}. Moreover, Theorem \ref{thm:main} is close to \cite[Theorem 11]{hanselman2017bordered}, which is a special case of \cite[Theorem 52]{hanselman2017bordered}.

\section*{Acknowledgements}
I would like to thank my advisor Jennifer Hom for suggesting this problem, and I cannot thank her enough for her continued support, guidance, and patience. I am also grateful to Wenzhao Chen and Tye Lidman for constructive comments on an earlier draft and to Steven Sivek for informative email correspondence.

\section{Preliminaries}\label{sec:background}
In this section, we primarily recapitulate some conventions and results in \cite{chen2019knot,geiges2009contact} and then set up several notations, in preparation for proving Theorems \ref{thm:second} and \ref{thm:main} in Section \ref{sec:proof}. 

We begin with a more thorough discussion about Theorem \ref{thm:wchen}. The 5-tuple in Theorem \ref{thm:wchen} is obtained from a genus-one doubly-pointed bordered Heegaard diagrams, $(\overline{\Sigma}, \{\alpha_1^a, \alpha_2^a\}, \beta, w,z)$, of $P$. This is done by viewing $\beta$, $w$, and $z$ as embedded in $\partial(S^1\times D^2)$ and identifying the pair of arcs $\{\alpha_1^a, \alpha_2^a\}$ with the longitude-meridian pair of $\partial(S^1\times D^2)$. See Figure \ref{fig:5tuple} for an example of the Mazur pattern. 

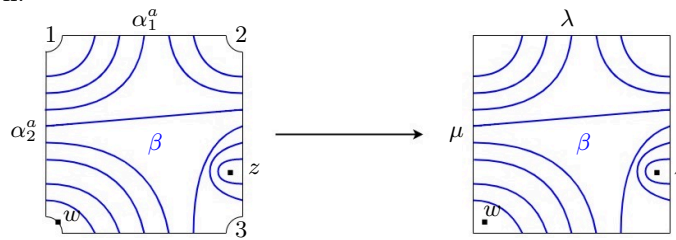
\begin{figure}[htbp]{\small
		\begin{overpic}[scale=0.31]
			{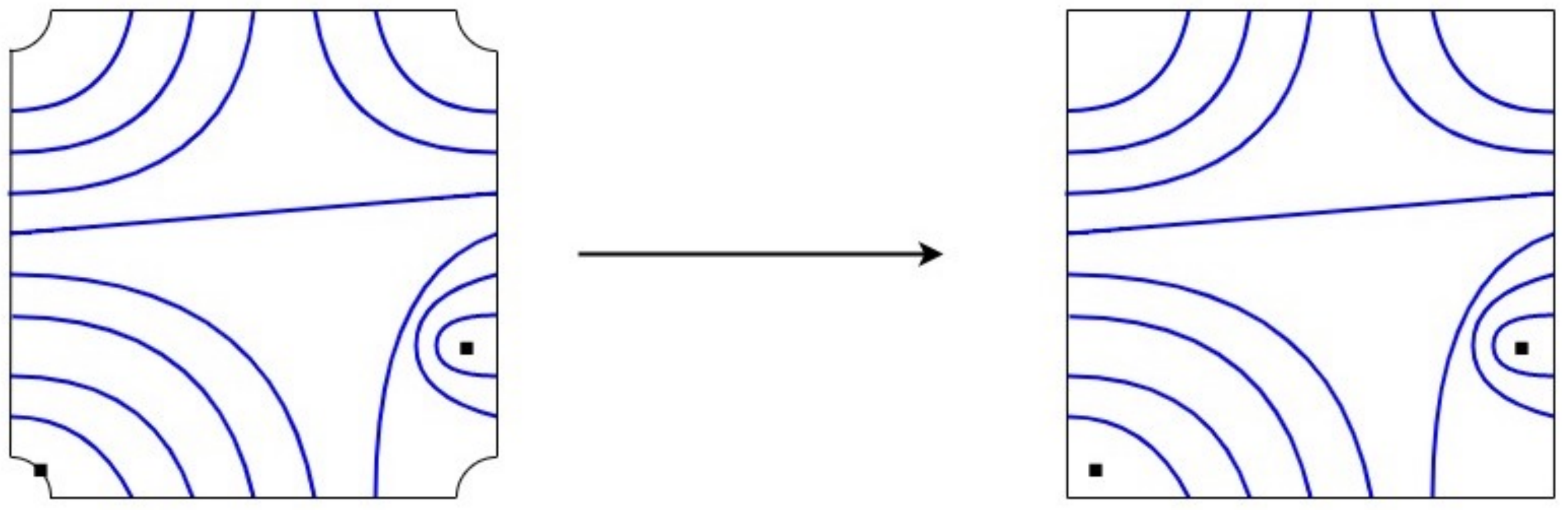}
			\put(1, 74){$1$}
			\put(73, 74){$2$}
			\put(73, 0){$3$}
			\put(8, 6){$w$}
			\put(78, 24){$z$}
			\put(34, 80){$\alpha_1^a$}
			\put(-12, 38){$\alpha_2^a$}
			\put(40, 33){\textcolor{blue}{$\beta$}}
			\put(196, 80){$\lambda$}
			\put(154, 38){$\mu$}
			\put(202, 33){\textcolor{blue}{$\beta$}}
			\put(167, 8){$w$}
			\put(239, 24){$z$}
	\end{overpic}}
	\caption{The data contained in a genus-one doubly-pointed bordered Heegaard diagram (on the left) of the Mazur pattern  can be equivalently understood as a 5-tuple (on the right). This convention comes from \cite[Sections~1~and~5]{chen2019knot}.}
	\label{fig:5tuple}
\end{figure}

In practice, Theorem \ref{thm:wchen} shows that, after identifying the torus $\partial(S^1\times D^2)$ with the quotient space $[0,1]\times[0,1]/\sim$ in the standard way and dividing the unit square evenly into four quadrants, we can fit $\alpha = \widehat{HF}(X_K)$ into the first quadrant (i.e., $[\frac{1}{2}, 1]\times [\frac{1}{2}, 1]$), fit $(\beta, w, z)$ into the third quadrant (i.e., $[0, \frac{1}{2}]\times [0, \frac{1}{2}]$), and extend them both horizontally and vertically to obtain a diagram that yields a chain complex isomorphic to $\widehat{CFK}(P(K))$. We call such a diagram a \textit{pairing diagram} for $P(K)$, and we denote by $\alpha(K)$ and $\beta(P)$ the curves obtained from $\alpha$ and $\beta$ by extension, respectively. Figure \ref{fig:pairingDiagram} displays four examples, in which $M$ denotes the Mazur pattern, and $T_{p,q}$ denotes the $(p,q)$-torus knot.

\begin{figure}[htbp]
    \begin{subfigure}[h]{0.45\linewidth}
        \centering
        \includegraphics[width=2.5cm]{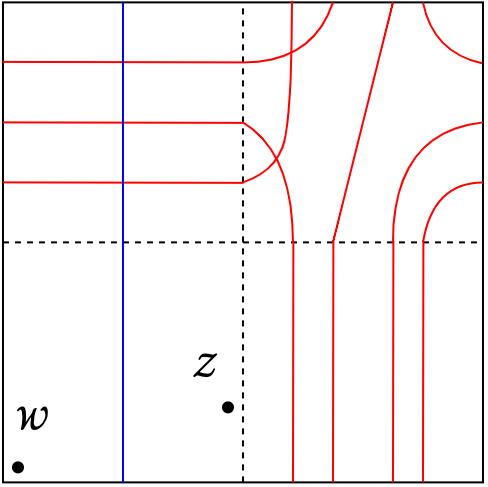}
        \caption{$U(T_{2,3})$}
    \end{subfigure}
    \begin{subfigure}[h]{0.45\linewidth}
        \centering
        \includegraphics[width=2.75cm]{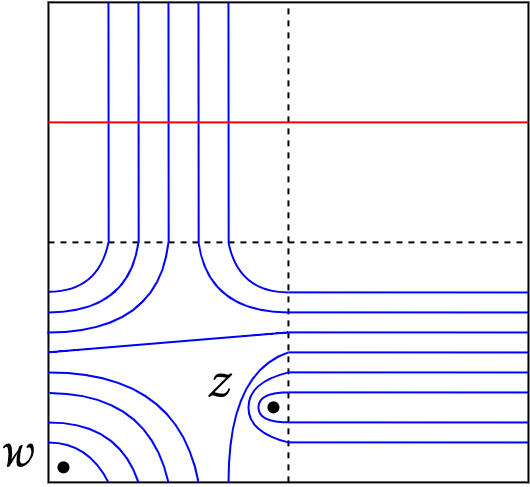}
        \caption{$M(U)$}
    \end{subfigure}
\vspace{3pt}
    \begin{subfigure}[h]{0.45\linewidth}
        \centering
        \includegraphics[width=2.75cm]{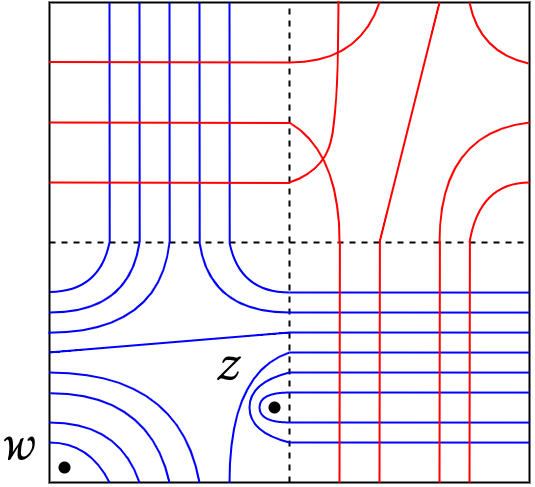}
     \caption{$M(T_{2,3})$}
    \end{subfigure}
    \begin{subfigure}[h]{0.45\linewidth}
        \centering
        \includegraphics[width=2.5cm]{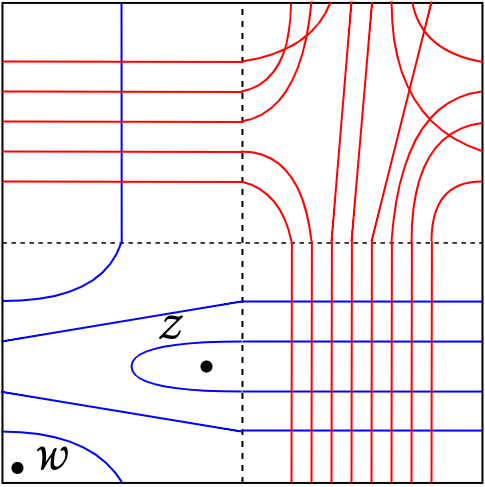}
        \caption{$(T_{2,5})_{3,1} := T_{3,1}(T_{2,5})$}
    \end{subfigure}
\caption{Examples of paring diagrams. The curves $\beta(P)$ and $\alpha(K)$ are drawn in blue and red, respectively.}
\label{fig:pairingDiagram}
\end{figure}

The proof of Theorem \ref{thm:second} needs some caution, as we will be moving $\alpha$-curves, which are immersed in general. The following lemma, which is widely known as the \textit{Whitney–Graustein theorem}, allows us to get rid of self-intersections of immersed curves (after certain modifications, which will be explained in Section \ref{sec:proof}). 

\begin{lemma}[\cite{whitney1937regular}, Theorem 1; see also \cite{geiges2009contact}, Theorem 1]\label{lemma:regularHomotopy}
    Regular homotopy classes of regular closed curves $\bar{\gamma}: S^1 \to \mathbb{R}^2$ are in one-to-one correspondence with the integers, the correspondence being given by 
    \[
        [\bar{\gamma}] \longmapsto \mathsf{rot}(\bar{\gamma}),
    \]
    where $\mathsf{rot}(\bar{\gamma})$ \footnote{By \cite{geiges2009contact}, the integer $\mathsf{rot}(\bar{\gamma})$ is called the \textit{rotation number} of $\bar{\gamma}$, and it is a signed count of the number of complete turns of the velocity vector $\bar{\gamma}'$ as we traverse $\bar{\gamma}$ in a pre-fixed orientation.} is the degree of the map $S^1 \to \mathbb{R}^2\backslash\{\mathbf{0}\}$, $s \mapsto \bar{\gamma}'(s)$.
\end{lemma}

The last thing we need to recall is how to obtain $\alpha$-curves from immersed curves in a (punctured) infinite cylinder, and vice versa. Given immersed curves in $(\mathbb{R}/(\frac{1}{2}+\mathbb{Z})) \times \mathbb{R}$, we place a grid system consisting of two vertical columns of unit squares, with the middle vertices identified with the punctured points of the cylinder. Then we follow the curve and replicate its segment in a square every time we meet an edge of a grid square. In this way we build its corresponding $\alpha$-curves. Likewise, if we start with a torus with $\alpha$-curves, we can trace the curves and recover its immersed curves in an infinite cylinder, as illustrated in Figure \ref{fig:cylinderSquare}.

\begin{figure}[htbp]{\small
		\begin{overpic}[scale=0.43]
			{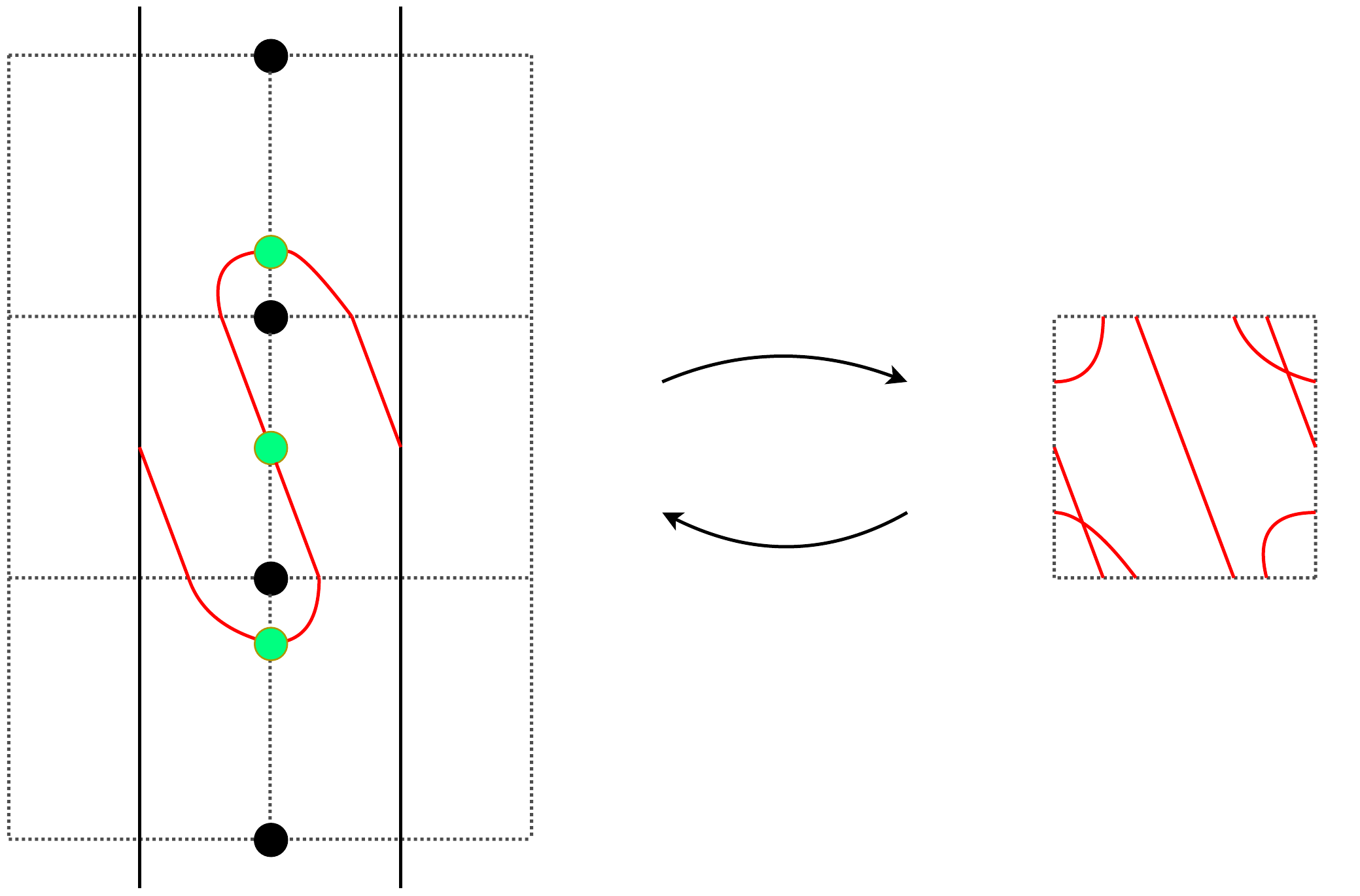}
			\put(17, 83){$0$}
			\put(32, 72){\textcircled{1}}
			\put(230, 73){\textcircled{1}}
			\put(33, 45){\textcircled{2}}
			\put(252, 96){\textcircled{2}}
			\put(61, 47){\textcircled{3}}
			\put(189, 98){\textcircled{3}}
			\put(59, 71){\textcircled{4}}
			\put(189, 83){\textcircled{4}}
			\put(34, 97){\textcircled{5}}
			\put(252, 82){\textcircled{5}}
			\put(33, 122){\textcircled{6}}
			\put(252, 68){\textcircled{6}}
			\put(60, 122){\textcircled{7}}
			\put(189, 68){\textcircled{7}}
			\put(60, 96){\textcircled{8}}
			\put(210, 92){\textcircled{8}}
			\put(18, -10){$-\frac{1}{2}$}
			\put(49, -10){$0$}
			\put(74, -10){$\frac{1}{2}$}
	\end{overpic}}
	\caption{The immersed curve for $-T_{2,3}$ in an infinite cylinder (left) and the $\alpha$-curves for $-T_{2,3}$ in a torus (right), where the circled numbers can be used to follow the construction.}\label{fig:cylinderSquare}
\end{figure}
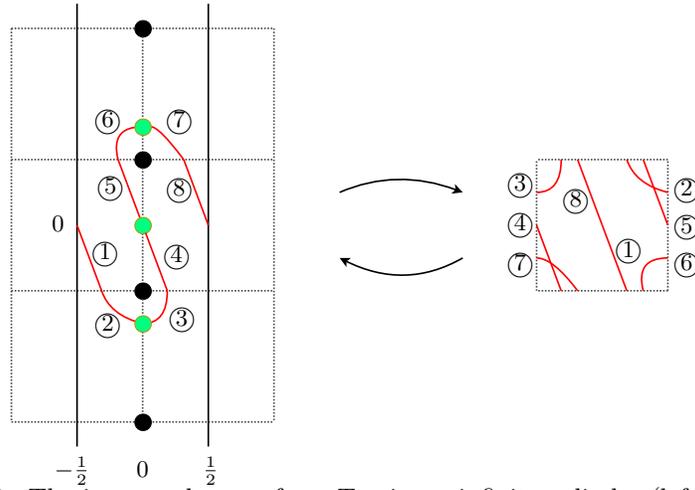

\section{Proof of Theorems}\label{sec:proof}
\begin{proof}[Proof of Theorem \ref{thm:second}]
Given a $(1,1)$-pattern $P \subset S^1\times D^2$ and a companion knot $K\subset S^3$, let $(T^2, \alpha(K), \beta(P), w, z)$ be a pairing diagram for $P(K)$, where $T^2 \cong S^1\times S^1$. Lift the diagram to $\mathbb{R}^2$ by the covering map $\pi:= p \times p: \mathbb{R}^2 \to T^2$, where $p:\mathbb{R}\to S^1$ is given by $x\mapsto (\cos(2\pi x), \sin(2\pi x))$. Let $\beta_0$ be a lift of $\pi^{-1}(\beta(P))$. By the construction of $\beta(P)$, the lift $\beta_0$ is connected. (See Figure \ref{fig:beta0forMazur} for an example of the Mazur pattern, where the lifts of extended $\alpha$-curves are omitted.) Let $\alpha_0$ be a lift of $\alpha(K)$. Notice that $\alpha_0$ may not be connected, for the immersed curves $\alpha$ may consist of multiple components. (See Figure \ref{fig:alpha0} for an example of the right-handed trefoil, where the lifts of extended $\beta$-curves are omitted.)

\begin{figure}[htbp]{\small
		\begin{overpic}[scale=0.42]
			{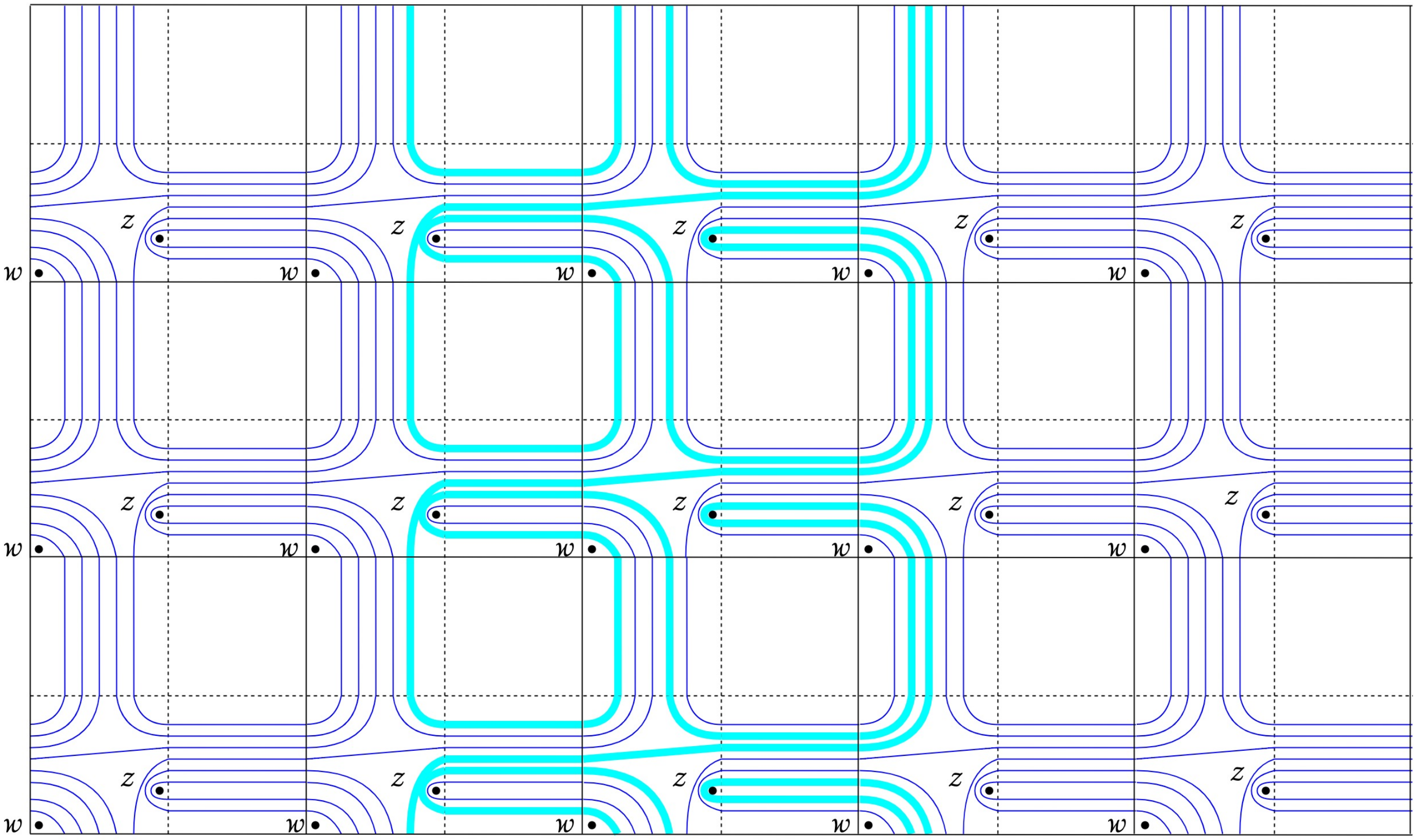}
	\end{overpic}}
	\caption{A choice of $\beta_0$ (highlighted in cyan) that corresponds to the Mazur pattern. (With a slight abuse of notation, we use the same symbols for the lifts of the two basepoints, respectively.)}\label{fig:beta0forMazur}
\end{figure}

\begin{figure}[htbp]{\small
		\begin{overpic}[scale=0.29]
			{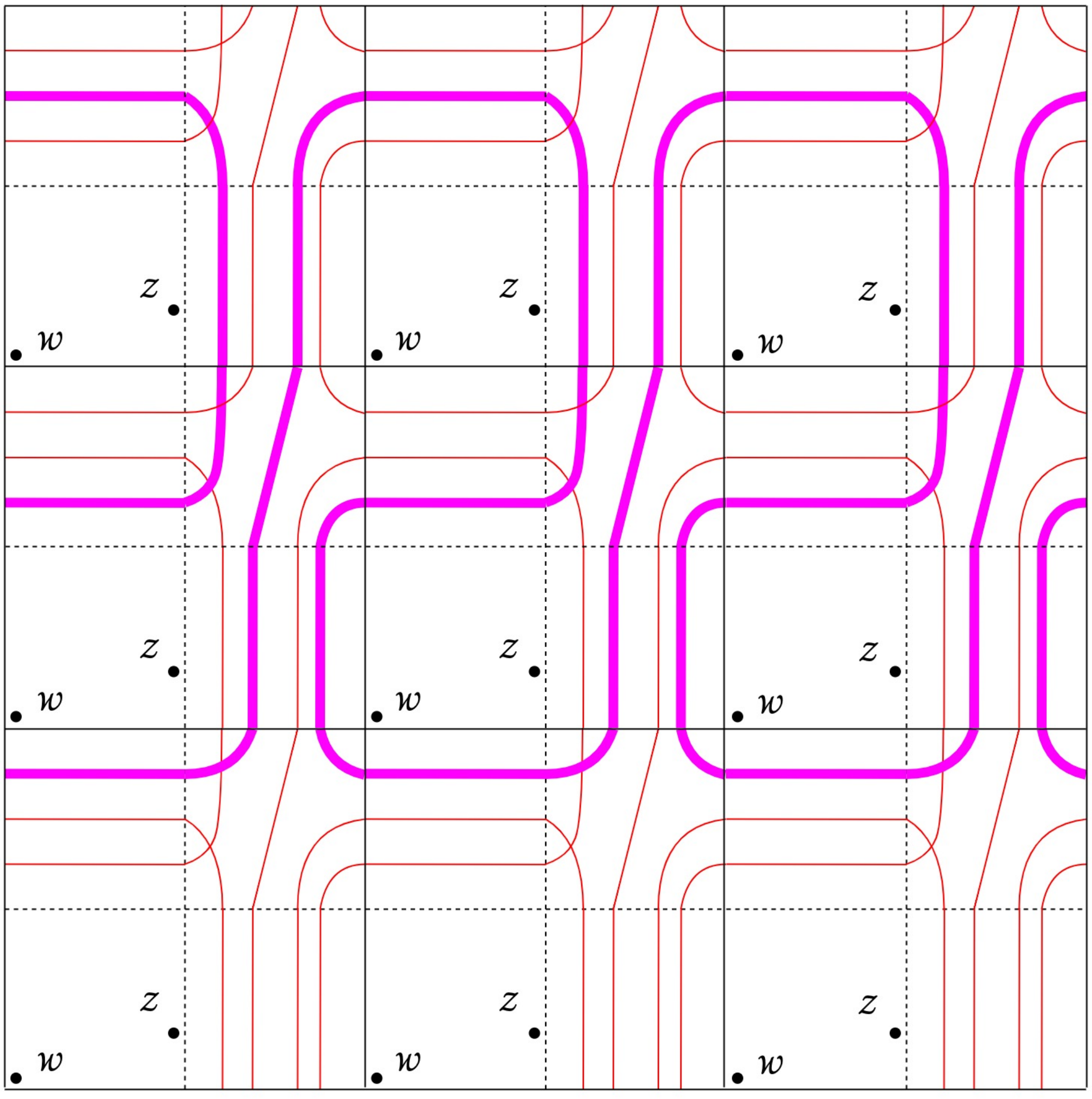}
	\end{overpic}}
	\caption{A choice of $\alpha_0$ (highlighted in magenta) that corresponds to $T_{2,3}$.}\label{fig:alpha0}
\end{figure}

By \cite[page 3]{hanselman2019cabling}, any immersed multi-curve in an infinite cylinder has a unique component wrapping around the cylinder. Then the construction in the end of Section \ref{sec:background} implies that there is at least one horizontal line segment in the second quadrant of $(T^2, \alpha(K), \beta(P), w, z)$. Thus, the lift $\alpha_0$ contains horizontal line segments. 

Ignore any closed component that $\alpha_0$ may contain. Then $\alpha_0$ becomes a connected piece going to the left- and right- infinity on $\mathbb{R}^2$. (Note that we may lose some data by doing so; nevertheless, we shall see by the end of the proof that inequality (\ref{eqn: pattern}) still holds.) Without loss of generality, give $\alpha_0$ an overall rightward orientation.
Then there exist rightward oriented horizontal line segments in $\alpha_0$; indeed, otherwise, $\alpha_0$ would not extend to the right-infinity. Denote by $\mu_1$ the first such segment that intersects $\beta_0$. Without loss of generality, suppose that $\mu_1 \subset [0,\frac{1}{2}]\times [\frac{1}{2},1]$. Consider the set $\mathcal{M}:=\pi^{-1}(\pi(\mu_1)) \cap ([1,\infty)\times [0,1])$, which consists of all lifts of $\pi(\mu_1)$ that lie in $[1,\infty)\times [0,1]$. For each element in $\mathcal{M}$, denote it by $\mu_j$ if it lies in $[j-1,j]\times [0,1]$. Let $k:= \max\{j \mid \mu_j \cap \beta_0 \neq \emptyset\}$, and let $\alpha_0^*$ denote the connected portion of $\alpha_0$ between $\mu_1$ and $\mu_{k+1}$, with $\mu_1$ included and $\mu_{k+1}$ excluded. (See Figure \ref{fig:gammas2.0} for an illustration, where the segments $\mu_j$'s are highlighted in green.)

\begin{figure}[htbp]{\small
		\begin{overpic}[scale=0.4]
			{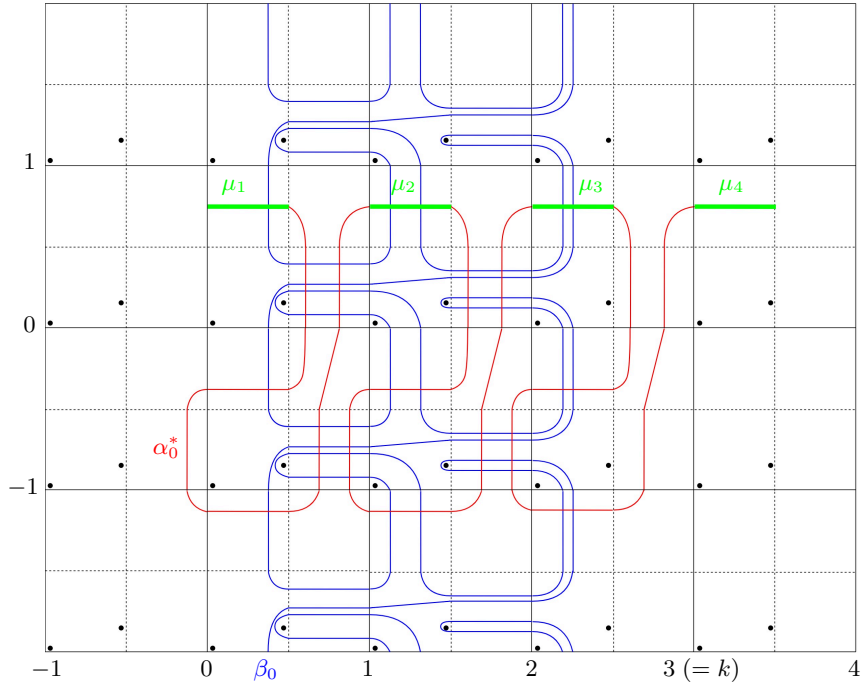}
			\put(256, 175){\textcolor{green}{$\mu_4$}}
			\put(203, 175){\textcolor{green}{$\mu_3$}}
			\put(132, 175){\textcolor{green}{$\mu_2$}}
			\put(68, 175){\textcolor{green}{$\mu_1$}}
			\put(80, -9){\textcolor{blue}{$\beta_0$}}
			\put(42, 75){\textcolor{red}{$\alpha_0^*$}}
			\put(-4, -9){$-1$}
			\put(60, -9){$0$}
			\put(121, -9){$1$}
			\put(183, -9){$2$}
			\put(235, -9){$3 \;(=k)$}
			\put(305, -9){$4$}
			\put(-13, 60){$-1$}
			\put(-7, 121){$0$}
			\put(-7, 183){$1$}
	\end{overpic}}
	\caption{An example of $M(T_{2,3})$.}
	\label{fig:gammas2.0}
\end{figure}

Notice that $\alpha_0$ is periodic, exhibiting horizontal translational symmetry, so $\alpha_0^*$ is periodic as well. Moreover, $\alpha_0^*$ consists of $k$ periods, with the $j$-th period of $\alpha_0^*$ starting from $\mu_j$ and terminates at the left endpoint of $\mu_{j+1}$, $j=1,2, \cdots, k$, if we traverse $\alpha_0^*$ from left to right.

As we mentioned in Section \ref{sec:background}, the curve $\alpha_0^*$ may contain self-intersections; we claim that all of them can be resolved by regular homotopies (which are allowed to cross any lift of basepoints). Indeed, we can complete $\alpha_0^*$ into an immersed closed curve by attaching the top endpoint of a left semi-circle of radius $R$ to the left endpoint of $\mu_1$, attaching the top endpoint of a right semi-circle of radius $R$ to the left endpoint of $\mu_{k+1}$, and then connect the two bottom endpoints of these two semi-circles by a line segment, where $R$ is sufficiently large so that the newly-added three segments do not intersect $\alpha_0^*$. By the $180^{\circ}$ symmetry of immersed curves for knot complements (up to regular homotopy), the turning number induced from each self-intersection will be canceled by that of its symmetric counterpart, so overall, the rotation number of the closed curve we just created is $\pm 1$, depending on the orientation of it. By Lemma \ref{lemma:regularHomotopy}, this rotation number is preserved under regular homotopies, so this closed curve is in the class of circles. Therefore, we can resolve all possible self-intersections of $\alpha_0^*$.

For each $j\in \{1,2,\cdots, k\}$, consider (the closure of) the complement of $\mu_j$ in the $j$-th period of $\alpha_0^*$. Allowing passing lifts of basepoints, regularly homotope this complement with its two endpoints fixed until the curve is within $[j-\frac{1}{2}, j]\times [\frac{1}{2},1]$. By the claim above, we may assume all self-intersections have been resolved, so we can further regularly homotope $\alpha_0^*$ until it becomes a horizontal line segment in $[0, k]\times [\frac{1}{2},1]$. We call the resulting curve $\alpha_0^{*'}$.

Since $\alpha(U)$ is a horizontal line segment, the $\alpha_0$-curve that corresponds to the unknot $U$, denoted by $\alpha_0(U)$, is a horizontal line in $\mathbb{R}\times [\frac{1}{2},1]$. Moreover, $\alpha_0(U)$ first intersects $\beta_0$ in the square $[0,1]\times[0,1]$ and lastly in $[k-1,k]\times[0,1]$, by the definitions of $\mu_1$ and $k$. Therefore, to get the dimension of $\widehat{HFK} (P(U))$, it suffices to consider $\beta_0$ and the portion of $\alpha_0(U)$ lying in $[0, k]\times [\frac{1}{2},1]$. We denote this portion by $\alpha_0^*(U)$, which is exactly the $\alpha_0^*$-curve that corresponds to $U$, and moreover, it can be identified with $\alpha_0^{*'}$.

Consider the set $\alpha_0^{*'}\cap \beta_0$ of intersection points. For each pair $x,y \in \alpha_0^{*'}\cap \beta_0$ that form the two vertices of a trivial bigon (i.e., a bigon that has no basepoint inside) between $\alpha_0^{*'}$ and $\beta_0$, we denote the bigon by $B_{x,y}$ and take a sufficiently small open neighborhood $U_{x,y} \supset B_{x,y}$ such that (1) no basepoint is inside, and (2) after we regularly homotope $\beta_0$ inside $U_{x,y}$ to eliminate the trivial bigon, no new intersection point is generated. Condition (2) can be achieved since we are considering finitely many segments in $\mathbb{R}^2$. 
If multiple bigons are nested, then we start with the innermost one, and in this order, condition (2) can still be achieved.
By the definition of $U_{x,y}$, each move of $\beta_0$ does not cross basepoints. Therefore, after all such trivial bigons are eliminated, we obtain a minimal intersection diagram between $\alpha_0^{*'}$ and $\beta_0$, and the final intersection number equals $\dim \widehat{HFK} (P(U))$. 

Observe the following: 1) the sequence of moves described above eliminate all trivial bigons generated by $\alpha_0^*$ and $\beta_0$ and does not increase the number of intersections with $\beta_0$, 2) the number of intersections between $\alpha_0^*$ and $\beta_0$ is at most the number of intersections between the original $\alpha_0$ and $\beta_0$, and 3) the minimum intersection number (obtained by eliminating trivial bigons via regular homotopies without passing basepoints) between the original $\alpha_0$ and $\beta_0$ gives $\dim \widehat{HFK} (P(K))$. These observations together with the result in the above paragraph imply that $\dim \widehat{HFK} (P(K)) \geqslant \dim \widehat{HFK} (P(U))$.
\end{proof}

\begin{remark}
In the proof above, we applied regular homotopies in the covering space $\mathbb{R}^2$ of $T^2$. Recall that the covering map $\pi$ is defined as $p \times p$, where $p:\mathbb{R}\to S^1$ is given by $x\mapsto (\cos(2\pi x), \sin(2\pi x))$. Composing those regular homotopies with $\pi$, we shall get regular homotopies in the base space $T^2$.
\end{remark}


\begin{proof}[Proof of Theorem \ref{thm:main}]
Here we continue with the curves $\alpha_0$ and $\beta_0$ that were set up in the first paragraph of the proof of Theorem \ref{thm:second}.

Recall that the 5-tuple $(\beta, \mu, \lambda, w, z) \subset \partial (S^1\times D^2)$ is constructed from a genus-one doubly-pointed bordered Heegaard diagram, say $(\overline{\Sigma}, \{\alpha_1^a, \alpha_2^a\}, \beta, w,z)$, for $(S^1\times D^2,P)$. Forgetting the $z$-basepoint, we obtain a genus-one bordered Heegaard diagram for the solid torus $S^1\times D^2$ with the standard parametrization of $\partial(S^1 \times D^2)$. Therefore, up to isotopy (not passing the $w$-basepoint), the diagram $(\overline{\Sigma}, \{\alpha_1^a, \alpha_2^a\}, \beta, w)$ is the one in Figure \ref{fig:borderedSolid}.

\vspace{2pt}
\begin{figure}[htbp]{\small
		\begin{overpic}[scale=0.33]
			{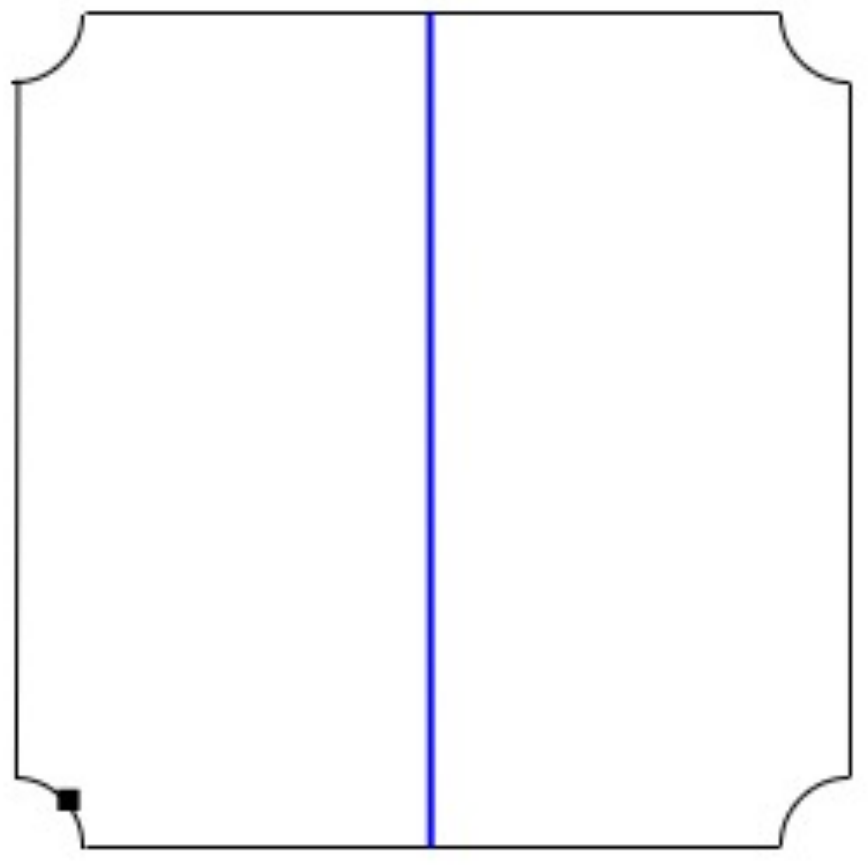}
			\put(-10, 38){$\alpha_2^a$}
			\put(38, 85){$\alpha_1^a$}
			\put(0, 78){$1$}
			\put(78, 78){$2$}
			\put(78, 0){$3$}
			\put(44, 38){\textcolor{blue}{$\beta$}}
			\put(7, 7){$w$}
	\end{overpic}}
	\caption{The genus-one bordered Heegaard diagram of $S^1\times D^2$, up to isotopy.}
	\label{fig:borderedSolid}
\end{figure}
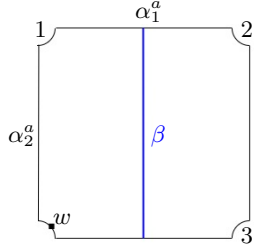

It then follows from the construction of the $\beta_0$-curve that, if we forget the basepoint $z$, we can isotope $\beta_0$ without passing any lift of the basepoint $w$ until it becomes a vertical straight line; we denote the line by $\beta_0'$. See Figure \ref{fig:passZ} for an illustration.

\begin{figure}[htbp]{\small
		\begin{overpic}[scale=0.5]
			{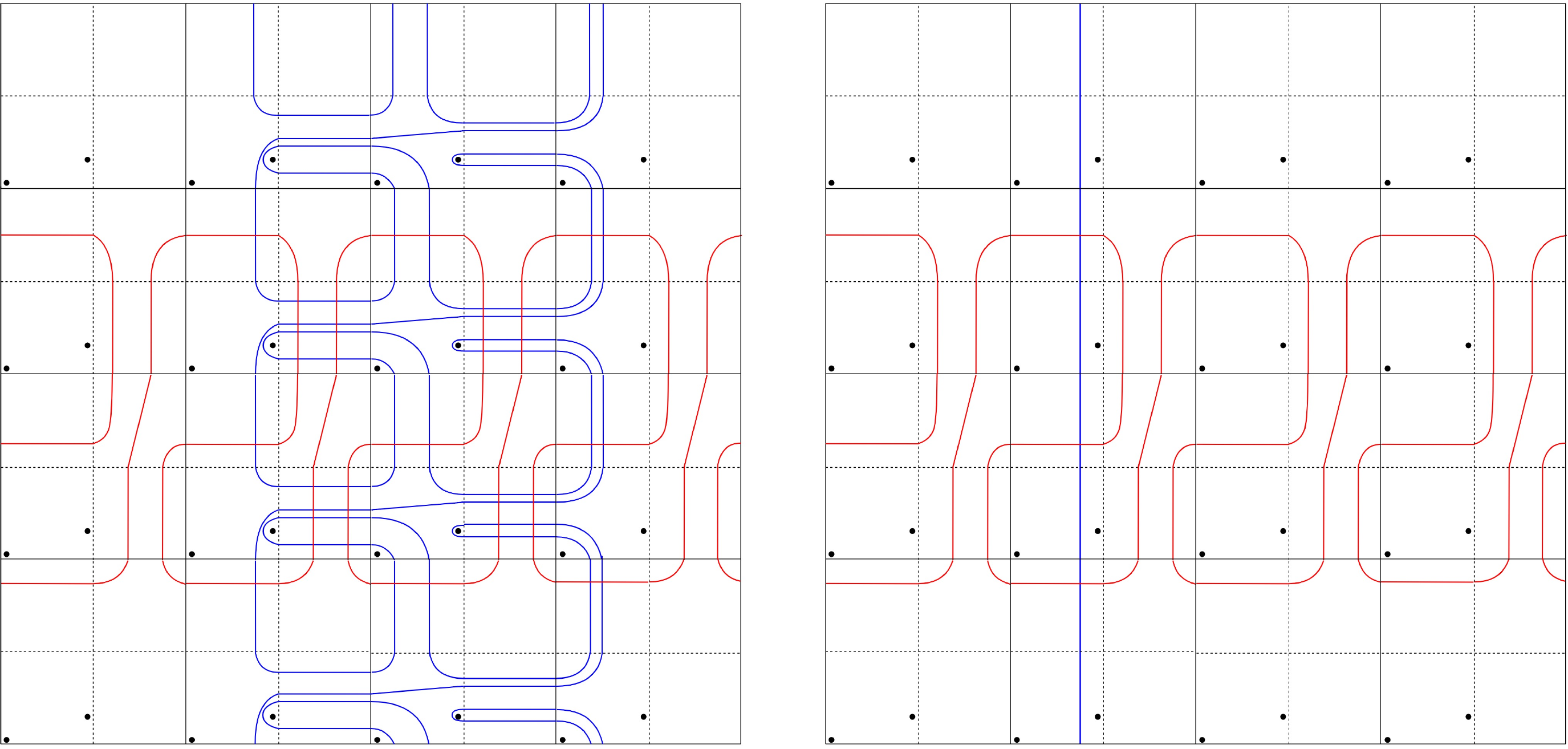}
			\put(61, -8){\textcolor{blue}{$\beta_0$}}
			\put(271, -8){\textcolor{blue}{$\beta_0'$}}
			\put(-12, 41){\textcolor{red}{$\alpha_0$}}
			\put(197, 41){\textcolor{red}{$\alpha_0$}}
	\end{overpic}}
	\caption{The $\alpha_0$-, $\beta_0$-, and $\beta_0'$-curves for $M(T_{2,3})$.}
	\label{fig:passZ}
\end{figure}
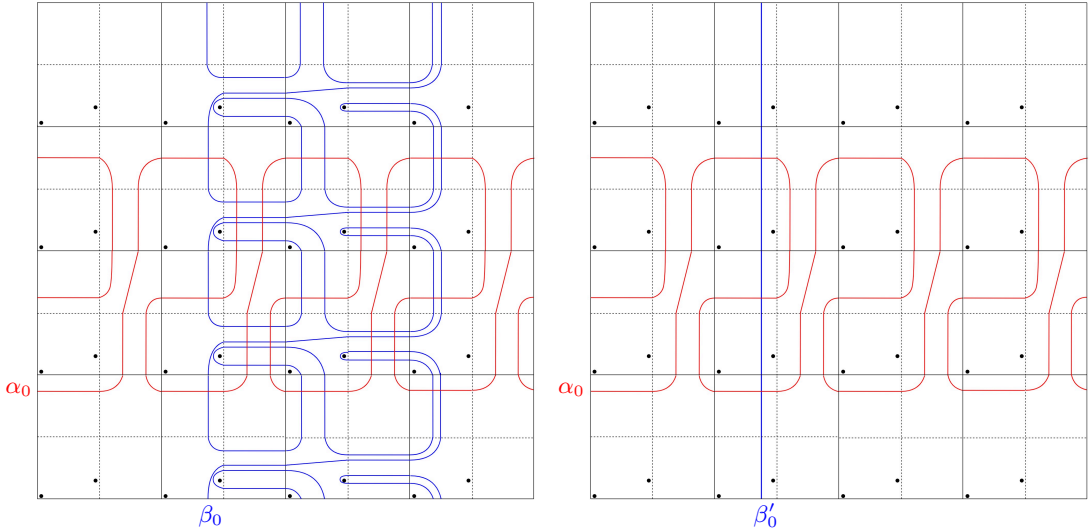

Since $\beta(U)$ is a vertical line segment, the $\beta_0$ curve that corresponds to the unknot pattern, denoted by $\beta_0(U)$, is a vertical line, which can be identified with $\beta_0'$. Since the number of intersections between $\beta_0(U)$ and $\alpha_0$ gives $\dim \widehat{HFK} (U(K))$, the number of intersections between $\beta_0'$ and $\alpha_0$ equals $\dim \widehat{HFK} (K)$.

Now we compare the pair of curves $(\alpha_0,\beta_0')$ with the original $(\alpha_0,\beta_0)$. Observe the following: 1) the isotopies described above eliminate all trivial bigons generated by $\beta_0$ and $\alpha_0$ and does not increase the number of intersections with $\alpha_0$, and 2) the minimum intersection number (obtained by eliminating trivial bigons via isotopies without passing basepoints) between $\beta_0$ and $\alpha_0$ gives $\dim \widehat{HFK} (P(K))$. These observations together with the result in the above paragraph imply that $\dim \widehat{HFK} (P(K)) \geqslant \dim \widehat{HFK} (K)$.

\end{proof}

\section{Further Remarks}\label{sec: FurtherRemarks}
In this section we make several remarks, which were suggested by Jennifer Hom and Tye Lidman, further discussing the two inequalities we have proved.

We begin with an easy observation. A major property of knot Floer homology is that it categorifies the Alexander polynomial $\Delta_K(t)$ of knots $K \subset S^3$ \cite{ozsvath2004holomorphic}:
\[
    \Delta_K(t) = \sum_{a\in \mathbb{Z}} \left[\sum_{m\in \mathbb{Z}} (-1)^m \dim \widehat{HFK}_m (K,a) \right] t^a.
\]
Also, there is a symmetry \cite[Proposition 3.10]{ozsvath2004holomorphic}:
\[
    \widehat{HFK}_m (K,a) \cong \widehat{HFK}_{m-2a} (K,-a).
\]
Since one of the characterizing conditions for Alexander polynomials is that 
\[
    \Delta_K(1) = 1,
\] 
it follows that the parity of $\dim \widehat{HFK} (K)$ is odd.

\begin{remark}
Ozsv\'{a}th-Szab\'{o} \cite{ozsvath2004holomorphicgenus} proved that if $\dim \widehat{HFK} (K) = 1$ then $K$ is the unknot; Hedden-Watson \cite[Corollary 8]{hedden2018geography} showed that if $\dim \widehat{HFK} (K) = 3$ then $K$ is a (left- or right-handed) trefoil (see also \cite[Corollary 1.5]{ghiggini2008knot}). From these two knot-detecting results, the example depicted in Figure \ref{fig:alex} (which shows that there exists a $(1,1)$-satellite $K$ with $\dim \widehat{HFK} (K) = 5$), and the above parity result we deduce that, if $\dim \widehat{HFK} (K) < 5$ then $K$ is not a $(1,1)$-satellite.
\end{remark}

Next, we discuss some special cases when we have a strict inequality:
\begin{remark}
In the proof of Theorem \ref{thm:second}, we mentioned that any immersed multi-curve in an infinite cylinder has a unique component wrapping around the cylinder. When the set of immersed curves corresponds to a knot complement, this component is an invariant of the knot, and furthermore, an invariant of the concordance class of the knot \cite[Proposition 2]{hanselman2019cabling}. It follows that all slice knots have a trivial such component as the unknot does; in terms of the notations we used in the proof above, with all closed components removed, $\alpha_0$ is a straight horizontal line. Moreover, for non-trivial slice knots, there are some additional closed components. Therefore, for non-trivial slice companion knot $K$, inequality (\ref{eqn: companion}) in Theorem \ref{thm:main} is strict:
\[
    \dim \widehat{HFK} (P(K)) > \dim \widehat{HFK} (K).
\]
In addition to that case, Petkova \cite[Lemma 7]{petkova2013cables} showed that the complex $CFK^-(K)$ for Floer homologically thin knots $K$ splits into exactly one staircase summand and possibly multiple square summands. Geometrically, a square summand is represented by a closed component in $\alpha_0$, so for Floer homologically thin knots $K$ containing a square summand in $CFK^-(K)$, the above strict inequality holds as well.
\end{remark}

Our last remark is related to gradings:

\begin{remark}
In the proof of Theorem \ref{thm:main}, we managed to isotope the curve $\beta_0$ in a desired way, passing only the $z$-lifts.
Then by Theorem \ref{thm:wchen}, if $\alpha$ is connected, and if we only consider Maslov gradings, what we have proved is also true. That is, if $\alpha$ is connected, the inequality
\[
    \sum_{a\in \mathbb{Z}}\dim \widehat{HFK}_{m} (P(K),a) \geqslant \sum_{a\in \mathbb{Z}}\dim \widehat{HFK}_{m} (K,a)
\]
holds for any Maslov grading $m \in \mathbb{Z}$.

On the other hand, the regular homotopies in the proof of Theorem \ref{thm:second} cannot achieve that property in general; for example, Figure \ref{fig:gammas2.0} shows that we cannot tighten the curve $\alpha_0$ corresponding to $T_{2,3}$ to a horizontal line without passing any $w$-lift. 

If we just consider Alexander gradings, we will not arrive at rank inequalities that work for all $(1,1)$-satellite. Indeed, considering the example depicted in Figure \ref{fig:alex} (see also \cite[Example 1.8]{juhasz2016concordance}), we can make the following observations:
\begin{center}
\begin{tabular}{ c|c|c|c } 

 $a$ & $\sum_{m \in \mathbb{Z}} \dim \widehat{HFK}_m(T_{2,3},a)$ & $\sum_{m \in \mathbb{Z}} \dim \widehat{HFK}_m((T_{2,3})_{2,3},a)$ &Observations \\[3pt]
 \hline
 $-2$ &0 &1 &$0<1$ \\ 
 $-1$ &1 &0 &$1>0$ \\ 
 $0$  &1 &1 &$1=1$ \\

\end{tabular}
\end{center}

\noindent The last column in the above table shows that Theorems \ref{thm:second} and \ref{thm:main} do not have refinements for Alexander gradings. 

\begin{figure}[htbp]{\small
		\begin{overpic}[scale=0.4]
			{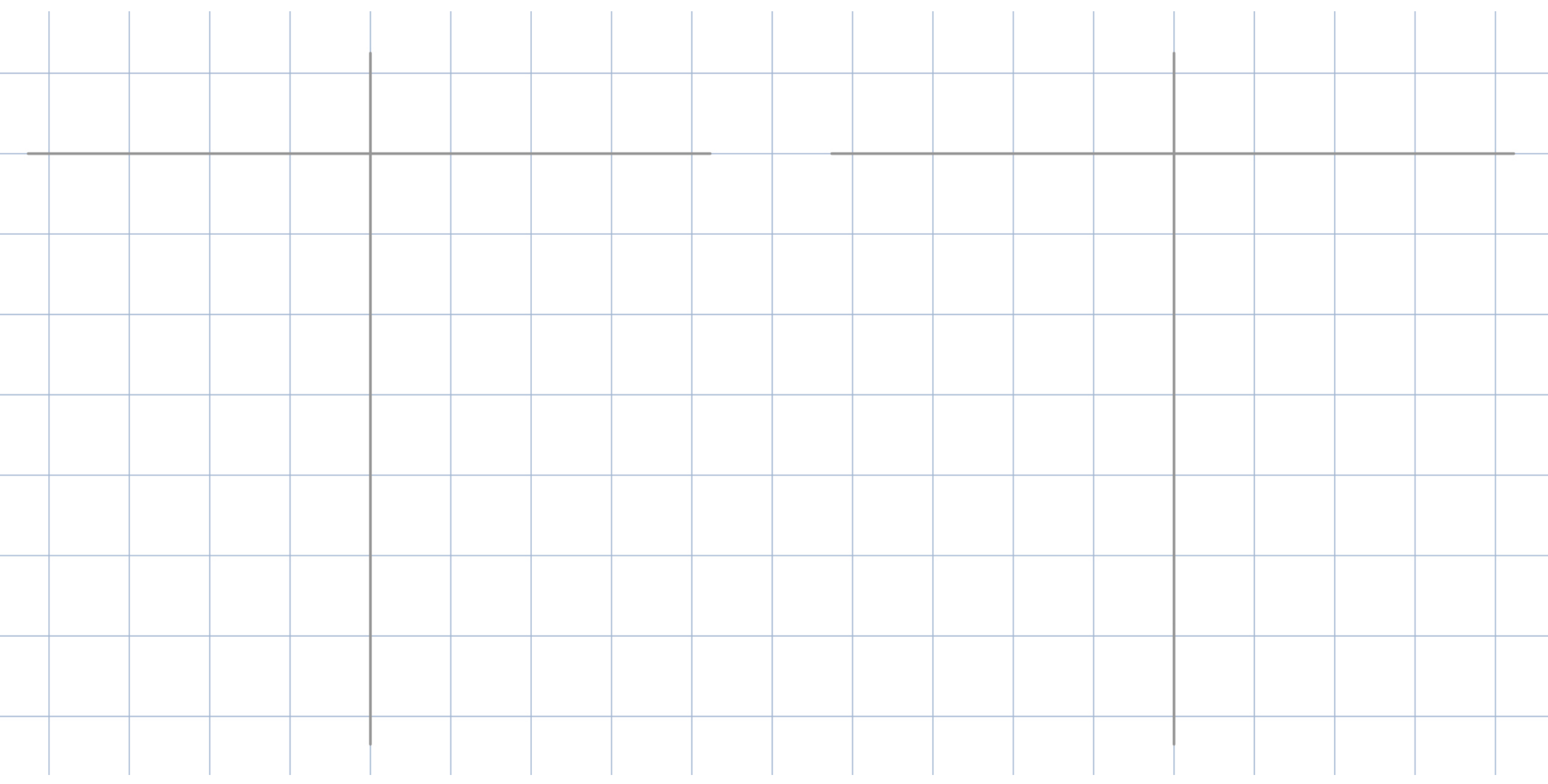}
			\put(56, 93){\textcolor{magenta}{$\mathbb{F}$}}
			\put(73, 110){\textcolor{magenta}{$\mathbb{F}$}}
			\put(90, 127){\textcolor{magenta}{$\mathbb{F}$}}
			\put(55, 132){$-1$}
			\put(78, 93){$-2$}
			\put(10, 70){\textcolor{magenta}{$\widehat{HFK}(T_{2,3})$}}
			\put(189, 10){\textcolor{magenta}{$\mathbb{F}$}}
			\put(205, 28){\textcolor{magenta}{$\mathbb{F}$}}
			\put(238, 95){\textcolor{magenta}{$\mathbb{F}$}}
			\put(272, 110){\textcolor{magenta}{$\mathbb{F}$}}
			\put(288, 127){\textcolor{magenta}{$\mathbb{F}$}}
			\put(185, 132){$-3$}
			\put(202, 132){$-2$}
			\put(272, 132){$2$}
			\put(244, 110){$-1$}
			\put(244, 28){$-5$}
			\put(244, 11){$-6$}
			\put(272, 70){\textcolor{magenta}{$\widehat{HFK}((T_{2,3})_{2,3})$}}
	\end{overpic}}
	\caption{Knot Floer homologies of $T_{2,3}$ and $(T_{2,3})_{2,3}$, plotted on the $(a,m)$-axis, respectively.}
	\label{fig:alex}
\end{figure}
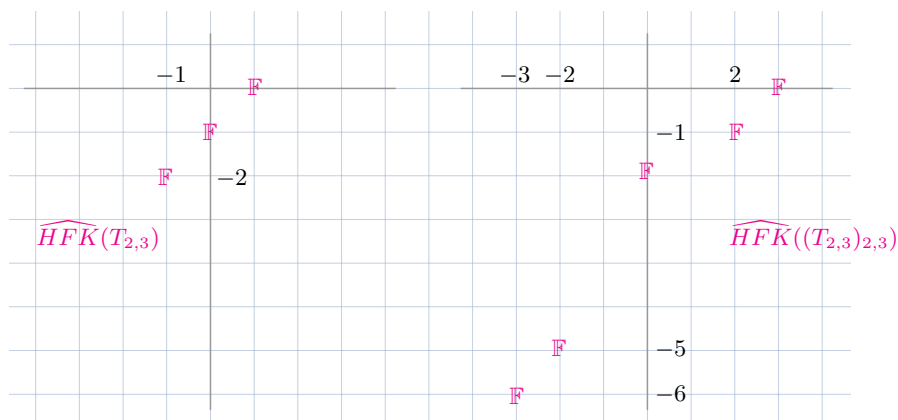

\end{remark}

\printbibliography 

\end{document}